\documentclass[11pt, reqno]{amsart}

\usepackage{tikz}
\usepackage[pagewise]{lineno}%\linenumbers

\usepackage[active]{srcltx}%%%% Inverse search

\usepackage[colorlinks=true,linkcolor=red,urlcolor=black]{hyperref} 

\newtheorem{theorem}{Theorem}[section]

\newtheorem{lemma}[theorem]{Lemma}
\newtheorem{proposition}[theorem]{Proposition}

\newtheorem{conjecture}[theorem]{Conjecture}

\newtheorem{maintheorem}{Theorem}

\theoremstyle{definition}

\newtheorem{remark}[theorem]{Remark}

\theoremstyle{remark}

\numberwithin{equation}{section}

\newcommand{\Z}{\mathbb Z}
\newcommand{\R}{\mathbb R}

\newcommand{\eps}{\varepsilon}

\newcommand{\eval}[2][\right]{\relax
  \ifx#1\right\relax \left.\fi#2#1\rvert}

\begin{document}
\title[Accessibility classes]{Structure of accessibility classes}

\author[Jana Rodriguez Hertz]{Jana Rodriguez Hertz}
\address{Jana Rodriguez Hertz,\newline
1. Department of Mathematics, 
Southern University of Science and Technology of China.
No 1088, xueyuan Rd., Xili, Nanshan District, Shenzhen,Guangdong, China 518055.}
\address{2. SUSTech International Center for Mathematics}
\email{rhertz@sustech.edu.cn}

\author[C. H. V\'asquez]{Carlos H. V\'asquez}
\address{Carlos H. V\'asquez, Instituto de Matem\'atica,
Pontificia Universidad Cat\'olica de Valpara\'{\i}so. Blanco Viel 596,
Cerro Bar\'on, Valpara\'{\i}so-Chile.} \email{carlos.vasquez@pucv.cl}

\begin{thanks}{CV  was supported by  Proyecto Fondecyt 1171427.}
\end{thanks}
\begin{thanks}
{JRH was partially supported by NSFC 11871262 and NSFC 11871394} 
\end{thanks}

\begin{abstract}
In this work we deal with dynamically coherent partially hyperbolic diffeomorphisms whose central direction is two dimensional. We prove that in general the accessibility classes are topologically immersed manifolds. If, furthermore, the diffeomorphism satisfies certain bunching condition, then the accessibility classes are immersed $C^{1}$-manifolds.  
\end{abstract}
\maketitle
\tableofcontents

\section{Introduction}
Let $M$ be a Riemannian closed manifold and let $f:M\to M$ be a partial hyperbolic diffeomorphism with an invariant splitting
$$TM=E^{s}\oplus E^{c}\oplus E^{u}.$$ 
Here  $E^{s}$ and  $E^{u}$  are uniformly hyperbolic bundles contracting and expanding, respectively, while  vectors in $E^{c}$ are neither contracted as strongly as the vectors in $E^{\rm s}$ nor expanded as the vectors in $E^{u}$. See the precise definition in Section \ref{section.preliminaries}.  It is well known that partial hyperbolicity is a $C^1$-open property.  Also it is known that 
there are unique invariant foliations $\mathcal{F}^{s}$ and $\mathcal{F}^{u}$ tangent to $E^{s}$ and $E^{u}$ respectively \cite{BP1974,HPS} but in general, $E^{c}$, $E^{cu} = E^{c} \oplus E^{u}$, and $E^{cs} = E^{c} \oplus E^{s}$ do not integrate to foliations, not even when $E^{c}$ is one-dimensional (see \cite{HHU2016}). A partially hyperbolic diffeomorphism is {\em dynamically coherent} if there are two invariant foliations: $\mathcal{F}^{cu}$ tangent to  $E^{c}\oplus E^{u}$, and $\mathcal{F}^{cs}$ tangent to $E^{c}\oplus E^{s}$. If $f$ is dynamically coherent, then there is an invariant foliation $\mathcal{F}^{c}$ tangent to $E^{c}$ (just take the intersection of $\mathcal{F}^{cs}$ and $\mathcal{F}^{cu}$). \par

We  say that a point $y\in M$  is $su$-{\em accessible  from} $x\in M$ if there exists a path $\gamma:I\to M$, from now on an $su$-{\em path}, piecewise contained   in the leaves of the strong stable and strong unstable foliations. This defines an equivalence relation on $M$. We denote by $AC(x)=\{y\in M: y\mbox{ is  }su-\mbox{accessible from }x\}$ the {\em accessibility class}  of $x$. A diffeomorphism has the {\em accessibility property} if there is  a unique accessibility class. We will mainly be working in the universal cover of $M$. When no confusion arises, $AC(x)$ will be denoting the accessibility class of $x$ in the universal cover which is not necessary the same as the preimage of $AC(x)$ by the covering projection from the universal cover.

The aim of this work is to describe the geometry of the accessibility classes:
 
 \begin{maintheorem}\label{mteo:class.manifold} If $f:M\to M$ is a dynamically coherent partially hyperbolic diffeomorphism with a two-dimensional center bundle, then the accessibility classes are topological immersed manifolds (see below). \par
 Moreover, if $AC(x)$ is not a codimension-one accessibility class, then either:
\begin{enumerate}
 \item $AC(x)$ is open, or else
 \item $AC(x)$ is a codimension-two immersed $C^{1}$-manifold. 
\end{enumerate}
\end{maintheorem}

Similars conclusion as above was obtained in \cite{FRH2005} and \cite{HS2017} in more restrictive setting than ours.

A topological immersed manifold is the image in $M$ of a topological manifold $N$ by a topological immersion $\phi$, that is by a continuous map such that for each point $x\in N$ there exists a neighborhood $U\subset N$ for which $\phi|_{U}$ is a homeomorphism onto its image. \par
With some additional hypotheses we can also obtain regularity of the codimension-one accessibility classes. 

\begin{maintheorem}\label{mteo:differentiability.AC} Let $f:M\to M$ be a dynamically coherent diffeomorphism with two-dimensional center bundle. Then all accessibility classes are injectively immersed $C^{1}$-submanifolds if any of the two following conditions hold:
\begin{enumerate}
\item $f$ is $C^{2}$ and satisfies the center bunching condition (\ref{bunching.condition}), 
\item  $f$ is $C^{1+\alpha}$ and satisfies the strong center bunching condition (\ref{strong.bunching.condition}).
\end{enumerate}
\end{maintheorem}

Theorem \ref{mteo:differentiability.AC} together with Proposition \ref{lower.semicontinuity} partially answer Problem 16 in \cite{HHU2007}

\begin{conjecture}[\cite{HHU2007}, Problem 16] Prove that the accessibility classes are topological manifolds that vary semi-continuously, as well as their dimensions. Prove that, with bunching, they are indeed smooth manifolds. 
\end{conjecture}

Theorem \ref{mteo:differentiability.AC}  also gives positive evidence of the following conjecture by Wilkinson:

\begin{conjecture}[Wilkinson \cite{W2013}, Conjecture 1.3] Let $f:M\to M$ be $C^{r}$, partially hyperbolic and $r$-bunched. Then the accessibility classes are injectively immersed $C^{r}$-submanifolds of $M$.
\end{conjecture}

Accessibility  is a key notion in the program of Pugh and Shub \cite{PS1995,PS2000} to prove that stable ergodicity is $C^r$-dense among volume preserving partially hyperbolic diffeomorphisms, $r\geq 2$.  Pugh and Shub also  conjectured the stable accessibility is dense among the $C^r$-partially hyperbolic diffeomorphisms, volume preserving or not, $r\geq 2$. In the case $\dim E^{\rm c}=1$, the accessibility property is always stable \cite{Didier2003}, and $C^\infty$-dense among the volume preserving diffeomorphisms \cite{HHU2008}. Even in the case of one-dimensional center bundle  the same result was proved in \cite{BHHTU2008} for non conservative diffeomorphisms. In the case $\dim E^{\rm c}=2$,  Avila and Viana \cite{AV} proved recently under bunching condition that if $f$ is accesible, then it is stable accesible. Without any hypothesis on the dimension of the central bundle,  Dolgopyat and Wilkinson \cite{DW2003} proved that stable accessibility is $C^1$- dense in the space of $C^r$-diffeomorphisms (any $r\ge 1$). \par

\noindent\textbf{Acknowledgment:} We want to thank R. Ures for helpful comments. We want to especially thank A. Wilkinson for a key remark, see Remark \ref{remark}. Also we would like to thank the referees for their valuable comments which helped to improve the manuscript.
%%%%%%%%%%%%%%%%%%%%%%%%%%%%%%%%%%%%%%%%%%%%%%%%%%%%%%%%%%%

\section{Preliminaries} \label{section.preliminaries}

Let $M$ be a Riemannian closed manifold. $f:M\to M$ is a {\em partially hyperbolic diffeomorphism} if there is a non trivial $Df$-invariant splitting of the tangent bundle
$$TM=E^{s}\oplus E^{c}\oplus E^{u}$$ 
and there are continuous positive functions defined  on $M$ satisfying for every $x\in M$:
$$\mu(x)<\nu(x)<\gamma(x)<\hat\gamma(x)^{-1}<\hat\nu(x)^{-1}<\hat\mu(x)^{-1}$$
with 
$$\max(\nu(x),\hat\nu(x))<1,$$ 
such that for every unit vector $v\in T_xM$
\begin{eqnarray}
\mu(x)<\|Df(x)v\|<\nu(x), & &\mbox{ for every } v\in E^{s}(x), \label{phcontraccion}\\
\gamma(x)<\|Df(x)v\|<\hat\gamma(x)^{-1}& & \mbox{ for every } v\in E^{c}(x),\label{phcenter}\\
\hat\nu(x)^{-1}<\|Df(x)v\|<\hat\mu(x)^{-1} ,& &\mbox{ for every } v\in E^{u}(x). \label{phexpansion}
\end{eqnarray}

The inequalities  \eqref{phcontraccion} and \eqref{phexpansion}  above mean that  $E^{s}$ and  $E^{ u}$  are uniformly hyperbolic bundles (contracting and expanding, respectively) while \eqref{phcenter}  means  vectors in $E^{c}$ are not contracted as strongly as the vectors in $E^{ s}$ nor expanded as strongly as the vectors in $E^{u}$.  It is well known that partial hyperbolicity is a $C^1$-open property. \par

We will say that $f$ satisfies the {\em center bunching condition} if the following holds: 
\begin{equation}\label{bunching.condition}
\max\{\nu(x),\hat\nu(x)\}\leq \gamma(x)\hat\gamma(x).
\end{equation}
When $f$ is a $C^2$ partially hyperbolic diffeomorphism dynamically coherent, then the (stable/unstable) holonomy maps are smooth restricted to the center stable/unstable leaves \cite{PSW1997}. Recently,  Brown \cite{brown2016} proved that the same statement hold if $f$ is  a   $C^{1+\alpha}$ partially hyperbolic diffeomorphism satisfying an stronger bunching condition. More precisely, take $0<\beta<\alpha$ such that 
$$\nu(x)\gamma(x)^{-1}<\mu(x)^{\beta}\quad\mbox{and}\quad \hat\nu(x)\hat\gamma(x)^{-1}<\hat\mu(x)^{\beta}.$$
we will say that $f$ satisfies the {\em strong bunching condition} if for some $0<\theta<\beta<\alpha$ such that 
$$\nu(x)^{\alpha}\gamma(x)^{-\alpha}\leq\nu(x)^{\theta}\quad\mbox{and}\quad \hat\nu(x)^{\alpha}\hat\gamma(x)^{-\alpha}\leq\hat\nu(x)^{\theta}$$
we have 

\begin{equation}\label{strong.bunching.condition}
 \max\{\nu(x),\hat\nu(x)\}^{\theta}\leq \gamma(x)\hat\gamma(x).
\end{equation}
 
%%%%%%%%%%%%%%%%%%%%%%%%%%%%%%%%%%%%%%%%%%%%%%%%%%%%%%%%%%%
\section{Proof of Theorem \ref{mteo:class.manifold}}

A subset $K\subset \mathbb{R}^n$ is said to be  \emph{ topologically homogeneous} if for every pair of points $x,y \in K$ there exist neighborhoods $U_x, U_y \subset  \mathbb{R}^n$ of $x$ and $y$ respectively and a homeomorphism $\varphi:U_x\to U_y$, such that $\varphi(U_x \cap K)=U_y \cap K$, and $\varphi(x)=y$, respectively. When $\varphi$ can be chosen a $C^r$-diffeomorphim, $r\geq1$, we say that  $K\subset \mathbb{R}^n$ is  \emph{$C^r$-homogeneous}. Let $K$ be a locally compact subset of $ \mathbb{R}^n$. \par

The following result holds for dynamically coherent partially hyperbolic diffeomorphisms with any center dimension. 

\begin{lemma}\label{c.homogeneous}
Let $x,y\in AC(x)$, and two center discs $D_{x}, D_{y}$ tangent to $E^{c}$ and centered, respectively, at $x$ and $y$. Then there exist two center discs $x\in U_{x}\subset D_{x}$ and $y\in U_{y}\subset D_{y}$ and a homeomorphism $h:U_{x}\to U_{y}$ such that  $h(x)=y$ and $h(\xi)\in AC(\xi)$ for all $\xi\in U_{x}$.
\end{lemma}
\begin{proof}
Let $x$ and $y$ be on the same $W^{s}_{\eps}(x)$, the argument then extends by applying the same argument a finite number of times, since for any pair of points $x,y$ in the same accessibility class there is a path $x_{0}=x,x_{1},\dots,x_{n}=y$ such that $x_{i}\in W^{s}_{\eps}(x_{i-1})$ or  $x_{i}\in W^{u}_{\eps}(x_{i-1})$, $i=1,\dots,n$. The composition of the $n$ local homeomorphisms gives us the desired local homeomorphism.\par 
Take $D_{x}$ and $D_{y}$, where $y\in W^{s}_{\eps}(x)$, then for any $\xi$ in a suitable $U_{x}\subset D_{x}$, there is a unique $w\in W^{s}_{\eps}(\xi)\cap W^{u}_{\eps}(D_{y})$. Let $h(\xi)=W^{u}_{\eps} (w)\cap D_{y}$. Clearly, $h(\xi)\in AC(\xi)$. Also, continuity of $\xi\mapsto W^{s}_{\eps}(\xi)$ and $\xi\mapsto W^{u}_{\eps}(\xi)$ implies that $h$ is continuous and open. The inverse of $h$ is defined analogously, taking $W^{u}_{\eps}$ instead of $W^{s}_{\eps}$ and viceversa. Therefore $h$ is a local homeomorphism. 
\end{proof}

\begin{remark}\label{remark}
The dynamically coherent hypothesis is essential in Lemma \ref{c.homogeneous}. If $f$ is not dynamically coherent, and we take any $c$-dimensional discs $D_{x}$, $D_{y}$ transverse to $E^{s}\oplus E^{u}$, then $W^{u}_{\eps}(D_{y})$ is not necessarily a manifold and there is no guarantee, a priori, that there is a unique point $w\in W^{s}_{\eps}(\xi)\cap W^{u}_{\eps}(D_{y})$.
\end{remark}

For every $\eps>0$ sufficiently small, in a sufficiently small ball $B(x)$ around $x$, and $D_{x}$ a 2-dimensional center disc, the following is well defined $\pi^{s}:B(x)\to W^{u}_{\eps}(D_{x})$, such that $\pi^{s}(y)=W^{s}_{\eps}(y)\cap W^{u}_{\eps}(D_{x})$. Analogously we define $\pi^{u}$. It is easy to verify that $\pi^{s}$ and $\pi^{u}$ are open maps.
Let us call $\pi^{su}=\pi^{s}\circ\pi^{u}:B(x)\to D_x$ and $\pi^{us}=\pi^{u}\circ\pi^{s}$ . Define:
$$F_{\eps}(x)=\{y\in D_{x}: (\pi^{us})^{-1}(y)\cap(\pi^{su})^{-1}(x)\ne\emptyset\}$$

Note that equivalently $F_{\eps}(x)=\pi^{us}((\pi^{su})^{-1}(x))$. Moreover, $F_{\eps}(x)$ is arc-connected (see Lemma \ref{lema.arco.conexo} below). We say that
$W^{s}$ and $W^{u}$ are {\em jointly integrable} at $x$ if $F_{\eps}(x)=\{x\}$. 

\begin{proposition}\label{teo.tricotomia}
If $\dim E^{c}=2$, then for each $x\in M$ and sufficiently small $\eps>0$, one of the following holds:
\begin{enumerate}
 \item $F_{\eps}(x)$ is a point,
 \item $F_{\eps}(x)$ is the injective image of a segment,
 \item the accessibility class of $x$, $AC(x)$, is open.
\end{enumerate}
\end{proposition}

\begin{remark}\label{Feps.pueden.ser.distintos}
{\em A priori} there is nothing that precludes the three different situations described in Proposition \ref{teo.tricotomia} above from happening in the same accessibility class. It is not known whether this is possible. 
\end{remark}

 See \cite{FRH2005} for the following lemma, we include it here with some suitable changes. 
 
\begin{lemma} \label{lema.arco.conexo}
Given $x\in M$, for any $y\in F_{\eps}(x)$, $y\ne x$, there exists an open disc $D'\subset D_{x}$ containing $x$ and a continuous $\psi: D'\times [0,1]\to D_{x}$ such that $\psi(x,0)=x$, $\psi(x,1)=y$ and $\psi(z,[0,1])\subset F_{\eps}(z)$.
\end{lemma}
 
\begin{proof}
 Take $y\in F_{\eps}(x)$, hence there exists $w\in (\pi^{us})^{-1}(y)\cap (\pi^{su})^{-1}(x)$. This implies that $\pi^{su}(w)=x$ and $\pi^{us}(w)=y$. Therefore, there exists a (two-legged) $su$-path in $M$ from $w$ to $y$, $\eta:[0,1]\to B(x)$. The projection $\pi^{su}\circ \eta$ gives a path in $D_{x}$ from $x$ to $y$ that is contained in $F_{\eps}(x)$. Take a disc $D_{w}$ transverse to $E^{s}_{w}\oplus E^u_{w}$. If $D'\subset D_{x}$ is sufficiently small, then for each $x'\in D'$ there is a unique $w'\in D_{w}$ such that $\pi^{su}(w')=x'$. By continuity of the stable and unstable foliations, we get close paths $\eta'$ for each $w'$, so we can choose a $\psi$ as in the statement. 
 \end{proof}

\begin{proof}[Proof of Proposition \ref{teo.tricotomia}] 
\begin{figure}[h]
\begin{tikzpicture}
\draw (0,0) circle (100pt);
\fill [color=gray] (-10pt,0pt) circle (20pt);
\fill [color=gray] (50pt,49pt) circle (20pt);
\draw (-100pt,0)..controls (-50pt, 20pt) and (20pt, -20pt).. (100pt,0);
\draw (-10pt,0)..controls (10pt,5pt)..(71pt, 71pt);
\draw (0,-50pt) node {$A$};
\draw (-20pt,50pt) node {$B_{1}$};
\draw (70pt,20pt) node {$B_{2}$};
\fill (-10pt,0pt) circle (2pt);
\draw (-10pt, -6pt) node {$x$};
\fill (50pt, 49pt) circle (2pt);
\draw (57pt, 40pt) node {$y$};
\draw (-10pt, 0pt) circle (20pt);
\draw (-15pt, -30pt) node {$D'$};
\draw (50pt,49pt) circle (20pt);
\draw (30pt, 70pt) node {$D''$};
\draw (70pt, -15pt) node {$\eta_{1}$};
\draw (80pt,80pt) node {$\eta_{2}$};
\end{tikzpicture}
\caption{\label{triod} A simple triod in $AC(x)\cap D$}
\end{figure}
From Lemma \ref{lema.arco.conexo}, it follows that $F_{\eps}(x)$ is path-connected.
If $F_{\eps}(x)$ is not a point, then it must contain an arc. Moreover, if $W^{s}$ and $W^{u}$ are not jointly integrable at $x$, then $F_{\eps}(x)$ always contains a segment that separates a small neighborhood of $x$. Indeed $\alpha=\pi^{u}(W^{s}_{\eps}(x))\subset F_{\eps}(x)$, and since $x$ is not an endpoint of $W^s_{\eps}(x)$ and $\pi^{u}$ is open, then $x$ is not an endpoint of the segment $\alpha$.\par

 Let us show that if $F_{\eps}(x)$ contains a simple triod, then $AC(x)$ must be open. A simple triod is a continuum homeomorphic to letter $Y$. If $F_{\eps}(x)$ contains a simple triod $Y$, then by Lemma \ref{c.homogeneous}, we may assume that $x$ is in the bifurcation point of the triod, as in Figure \ref{triod}; that is, there is an arc  $\eta_{1}$ separating a disc $U_{x}$ transverse to $E^{s}_{x}\oplus E^{u}_{x}$ into two connected components $A$ and $B$, and $\eta_{2}$ separating $B$ into two connected components $B_{1}$ and $B_{2}$, such that $\eta_{1}\cup \eta_{2}=Y\subset F_{\eps}(x)$, and such that $\overline{A}\cap \overline{B_{1}}\cap\overline{B_{2}}\supset\{x\}$. Notice that $\eta_{1}\cup \eta_{2}\subset AC(x)$. Let $y\in \overline{B_{1}}\cap \overline{B_{2}}$. By Lemma \ref{lema.arco.conexo}, there exists $\psi:D'\times [0,1]\to D_{x}$ such that $\psi(x,1)=y$ and $\psi(z,t)\subset F_{\eps}(z)\subset AC(z)$. For any center disc $D''$ centered at $y$, $D'$ can be chosen so that $\psi(D',1)\subset D''$. By continuity of $\psi$, and since we are in a plane disc, for all points $z$ in $A\cap D'$, we have $\psi(z,[0,1])\cap \eta_{1}\ne \emptyset$. This implies that $AC(z)\cap AC(x)\ne\emptyset$ for all $z\in A\cap D'$, hence $A\cap D'\subset AC(x)$. 
By continuity of $\pi^{su}$, and the fact that $D'$ is a 2-disc, $(\pi^{su})^{-1}(A\cap D')$ is an open set in $AC(x)$. By the center homogeneity of $AC(x)$ (Lemma \ref{c.homogeneous}), $AC(x)\cap D_{y}$ is open for every small disc transverse to $E^{s}_{y}\oplus E^u_{y}$, and hence $AC(x)$ is open. For another proof of the openness of $AC(x)$, see [HHU2008]. \par
This shows that the existence of a triod in $F_{\eps}(x)$ implies $AC(x)$ is open. Therefore, if $AC(x)$ is not open nor a point, $F_{\eps}(x)$ is a path connected, locally connected set without triods, therefore, it is the injective image of an arc. 
\par
\end{proof}

\begin{remark}\label{remark.triod} The proof above also establishes that, given a continuous map $\psi:D'\times[0,1]\to D_{x}$ as described in Lemma \ref{lema.arco.conexo}, if there are two points $y,z$ such that $\psi(y,[0,1])\cup\psi(z,[0,1])$ contains a simple triod $Y$, then $AC(y)=AC(z)$ is open. Note that we are not requiring that the $\psi$-image of close points form triods. This will follow from continuity of $\psi$ for some close points. The proof follows exactly as above. \end{remark}

%Let $$E_{\eps}(x)=(\pi^{su})^{-1}(x)\cup\bigcup \{(\pi^{us})^{-1}(y):(\pi^{us})^{-1}(y)\cap (\pi^{su})^{-1}(x)\ne\emptyset\}$$
% Hence, $F_{\eps}(x)=E_{\eps}(x)\cap D_{x}$. We have:
% 
%\begin{corollary}
% For each $x\in M$, and sufficiently small $\eps>0$, $E_{\eps}(x)$ is a topological manifold.
%\end{corollary}
% 
%It is easy to verify the following:
%
%\begin{lemma}\label{cadena.e}
%For any $y\in AC(x)$ there exist a finite number of points $x=x_{0},x_{1},\dots, x_{n}=y$ such that $x_{i}\in E_{\eps}(x_{i-1})$ for all $i=1,\dots, n$. 
%\end{lemma}

For any disc center disc $D=D_x$ containing $x$, let
$$AC^{D}(x)=\text{cc}(AC(x)\cap D,x),$$
that is, the connected component of $AC(x)\cap D$ that contains $x$. For the sake of notation we will omit the dependence of $D$ from $x$ when this when this does not cause confusion.
We have the following :
\begin{proposition}\label{teo.tricotomia.ac}
For each $x\in M$ and any sufficiently small center disc $D$ containing $x$, one of the following holds:
\begin{enumerate}
 \item $AC^{D}(x)$ is a point,
 \item $AC^{D}(x)$ is the injective image of a segment,
 \item the accessibility class of $x$, $AC(x)$, is open.
\end{enumerate}
\end{proposition}

\begin{remark}
 Even though in principle we could have points $x,y,z\in AC(z)$ such that $F_{\eps}(x)=\{x\}$, $F_{\eps}(y)$ is an arc and $AC(z)$ is open, as stated in Remark \ref{Feps.pueden.ser.distintos}, all sets $AC^{D}(x), AC^{D}(y)$ and $AC^{D}(z)$ are homeomorphic if $x,y,z$ belong to the same accessibility class. This is a consequence of Lemma \ref{c.homogeneous}.\par
 Note that, while $F_{\eps}(x)$ is arc connected, this could be not the case, a priori, for $AC^{D}(x)$. This fact will follow only after the proof below. \end{remark}

\begin{proof}

If $AC^{D}(x)$ is a point then, by Lemma \ref{c.homogeneous}, $AC^{D}(\xi)$ is a point for every $\xi\in AC(x)$. This happens if and only if the stable and unstable foliations are jointly integrable at all points in $AC(x)$, that is, if $F_{\eps}(\xi)=\{\xi\}$ for every $\xi\in AC(x)$. (Moreover, in this case $AC(x)$ is a $C^{1}$-manifold by Journ\'ee's argument, see \cite{Didier2003}, \cite{BHHTU2008}). \par

So let us assume that $AC(x)$ is not open and $AC^{D}(x)$ is not a point. Then by the previous paragraph and Proposition \ref{teo.tricotomia}, there exists a point $\xi\in AC(x)$ such that $F_{\eps}(\xi)$ is the injective continuous image of a segment for some $\xi\in AC(x)$. By the the remark at the beginning of the proof of Proposition \ref{teo.tricotomia}, $\alpha= F_{\eps}(\xi)$ separates a small center disc $D$, centered at $\xi$. 

%\begin{figure}[h]
% \includegraphics[width=8cm]{figure2}
% \caption{\label{figure2} A continuum $K$ not included in $\alpha$}
%\end{figure}

\begin{figure}[h]
\begin{tikzpicture}
\draw (0,0) circle (100pt);
\fill [color=gray] (-10pt,0pt) circle (20pt);
%\fill [color=gray] (50pt,-6pt) circle (20pt);
\draw (-100pt,0)..controls (-50pt, 20pt) and (20pt, -20pt).. (100pt,0);
\draw[color=red, line width=2pt] (-25pt,5pt)..controls (0pt, 7pt) and (25pt, -3pt)..(55pt, 0pt);
\draw (-10pt,0)..controls (0pt,30pt)..(71pt, 71pt);

\draw (5pt,-20pt) node {$A$};
%\draw (-25pt,25pt) node {$B$};
%\draw (15pt,15pt) node {$C$};
\fill (-10pt,0pt) circle (2pt);
\draw (-10pt, -7pt) node {$\xi$};
\fill (-30pt,3pt) circle (2pt); 
\draw (-38pt,-5pt) node {$z_{0}$};
\fill (50pt, -6pt) circle (2pt);
\draw (57pt, -12pt) node {$y$};
\draw (-10pt, 0pt) circle (20pt);
\draw (-15pt, -30pt) node {$D'$};
\draw (-20pt,30pt) node {$A'$};
%\draw (50pt,-6pt) circle (20pt);
%\draw (50pt, -36pt) node {$D''$};
\draw (-115pt, 5pt) node {$\alpha$};
\draw (85pt,85pt) node {$K$};
\end{tikzpicture}
\caption{\label{figure2} $AC^{D}(\xi)$ separates $D$ into at least 3 connected components}
\end{figure}

Let us suppose that $AC^{D}(\xi)$ contains a point that is not contained in $\alpha$. $AC^{D}(\xi)$ might be very complex, and the complement of $AC^{D}(\xi)$ in $D$ could consist of infinitely many connected components. However, due to the previous remark, no matter how small we choose a center disc $D$ containing $\xi$, $AC^{D}(\xi)$ will be a continuum in $D$ which contains points outside $\alpha$. \par

Take $y\in \alpha\setminus \{\xi\}$. Then, due to Lemmas \ref{lema.arco.conexo} and \ref{c.homogeneous} 
there exist a neighborhood $D'$ of $\xi$ and $\psi:D'\times [0,1]\to D$ continuous such that $\psi(\xi,0)=\xi$, $\psi(\xi,1)=y$, and $\psi(z,[0,1])\subset AC(z)$ for all $z\in D'$. We lose no generality in assuming $D'$ is closed. \par

Let $A$ and $A'$ the the two connected components of $D'\setminus\alpha$.  Let $D''$ be a closed disc containing $\xi$ in its interior such that $D''$ contained in $D'$, and so that $\partial D'$ and $\partial D''$ are very close. Let $K=AC^{D''}(\xi)$. Suppose there exists $\eta\in A'\cap K$. 
Let $z_{0}\in\alpha\cap\partial D'$ be such that $\xi\in[z_{0},y]\subset \alpha$, as depicted in Figure \ref{figure2}. Note that, by Remark \ref{remark.triod} and continuity, $\xi\in\psi(z_{0},[0,1])\subset \alpha$, and $\psi(z_{0},1)$ is close to $y$. Then there exists $\eps>0$ such that for all $z\in B_{\eps}(z_{0})\cap A'\subset D'\setminus D''$, we have that $\psi(z,[0,1])\cap K\ne\emptyset$. Indeed, if $\eps>0$ is sufficiently small, then $\psi(z,[0,1])$ is arbitrarily close to $\alpha$, by continuity of $\psi$. If $\psi(z,[0,1])$ did not meet $K$, then it would separate the ball $D''$ into two connected components, one containing $\eta\in K$ and the other containing $\alpha\cap D''\subset K$, which is a contradiction, since $K$ is a continuum. Therefore, $B_{\eps}(z_{0})\cap A'\subset AC^{D'}(\xi)$, since $\psi(z,[0,1])\subset AC(z)$, and $AC(z)=AC(\xi)$ if they intersect. \par
Since $B_{\eps}(z_{0})\cap A'$ contains an open set, this implies $AC^{D}(z)$ is open for some $z\in AC(\xi)=AC(x)$ and a sufficiently small $D$ containing $z$. From Lemma \ref{c.homogeneous} we get that $AC^{D}(z)$ is open for every $z\in AC(x)$, and therefore $AC(x)$ is open. 
\end{proof}

Note that in Propositions \ref{teo.tricotomia} and \ref{teo.tricotomia.ac} the bi-dimensionality of $E^{c}$ is strongly used.

\begin{theorem}
$AC(x)$ is a topological immersed manifold for all $x$. Moreover, if, $AC^{D}(\xi)$ is a point for all $\xi\in AC(x)$, then $AC(x)$ is an immersed $C^{1}$-manifold.
\end{theorem}

\proof Due to the center homogeneity proved in Lemma \ref{c.homogeneous}, all points in an accessibility class $AC(x)$ satisfy either (1), (2) or (3) in  Proposition~\ref{teo.tricotomia.ac} above.  
If $AC(x)$ is open, then the statement is obvious. \par
If for all $\xi\in AC(x)$, $AC^{D}(\xi)$ is a point, then $W^{s}$ and $W^{u}$ are jointly integrable at $\xi$ for all $\xi\in AC(x)$. The proof that $AC(x)$ is a topological manifold in this case, can be found in \cite{HHU2008}, Lemma A.4.1. and the discussion above it. Even though the context of this lemma is for one-dimensional center bundle, its proof applies, since the only thing that it is used is the joint integrability, namely, that $\pi^{su}(AC_{x}(\xi))$ is just one point, where $AC_{x}(\xi)$ is the connected component of $AC(\xi)\cap B_{\eps}(x)$ containing $\xi$. This holds due to joint integrability. \par
To see that in this case we also have that the immersed manifold is $C^{1}$, observe that $W^{s}(\eta)$ and $W^{u}(\eta)$ are, restricted to $AC(x)$, continuous transverse foliations with uniformly smooth leaves. Journ\'e's argument \cite{J1988} (see also Theorem \ref{journe}) then implies that $AC(x)$ is an immersed $C^{1}$-manifold. \newline\par

Let us now assume that $AC^{D}(\xi)$ is a segment for all $\xi \in AC(x)$. Take a small ball around $x$ and foliate it by center discs. Call this foliated ball $B(x)$. We may assume that $B(x)$ is small enough that $\pi^{us}$ is well-defined, where $\pi^{us}:B(x)\to D_{x}$ is defined as above Proposition \ref{teo.tricotomia}, and $D_{x}\subseteq B(x)$ is the center disc containing $x$.  We will see that the connected component $AC^{B}(x)$ of $AC(x)\cap B(x)$ that contains $x$ is a topological manifold. 

For each $y\in W^{s}_{\eps}(x)$, let $h_{y}:U_{x}\to U_{y}$ be the homeomorphism given by Lemma \ref{c.homogeneous}, where $U_{x}\subset D_{x}, U_{y}$ are discs in the two-dimensional foliation chosen above. Note that $y\mapsto h_{y}$ is continuous in the $C^{0}$-topology. Also note that 
$$h_{y}(AC^{U_{x}}(x))=AC^{U_{y}}(y)=cc(AC(x)\cap U_{y}, y).$$ 

The set  $Y=\bigcup\{h_{y}(AC^{U_{x}}(x)):y\in W^{s}_{\eps}(x)\}$ is an $(s+1)$-topological manifold. It can be parametrized by $[-1,1]\times B^{s}_{\eps}(0)$, where $[-1,1]$ is a parameter for $AC^{U_{x}}(x)$. This topological manifold is contained in $W^{cs}(x)$, which is transverse to $E^{u}$. Now by taking the local unstable set of each point in $Y$, we obtain an $(s+u+1)$-topological manifold $Z$ parametrized by $[-1,1]\times B^{s}_{\eps}(0)\times B^{u}_{\eps}(0)$. We claim that $Z$ is $AC^{B}(x)$, where
$$AC^{B}(x)=cc(AC(x)\cap B(x), x).$$
If this were not the case, then $\pi^{us}(AC^{B}(x))$ would contain more than just $AC^{U_{x}}(x)$, and this would be a contradiction, since $\pi^{us}(AC^{B}(x))$ is a connected set containing $x$ and contained in $AC(x)\cap U_{x}$. 
\endproof

This finishes the proof of Theorem \ref{mteo:class.manifold}.

% In \cite{RSS1996} the authors shown that $K$ is $C^1$-homogeneous if and only if $K$ is a $C^1$-submanifold of $ \mathbb{R}^n$. This fact, allows us to conclude the following statement.

%%%%%%%%%%%%%%%%%%%%%%%%%%%%%%%%%%%%%%%%%%%%%%%%%%%%%%%%%%%%%%
\section{Proof of Theorem \ref{mteo:differentiability.AC}}\label{section.theorem.B}

%A partially hyperbolic diffeomorphism is {\em dynamically coherent} if there are two invariant foliations: $\mathcal{F}^{cu}$ tangent to  $E^{c}\oplus E^{u}$, and $\mathcal{F}^{cs}$ tangent to $E^{c}\oplus E^{s}$. If $f$ is dynamically coherent, then there is an invariant foliation $\mathcal{F}^{c}$ tangent to $E^{c}$ (just take the intersection of $\mathcal{F}^{cs}$ and $\mathcal{F}^{cu}$). \par

Throughout this section we will assume that $f$ is a dynamically coherent partially hyperbolic diffeomorphism. \newline\par

The following three results hold for any center dimension.
\begin{theorem}\cite{PSW1997}\label{PSW} Suppose that $f:M\to M$ is a $C^{2}$ partially hyperbolic diffeomorphism that is dynamically coherent and satisfies the center bunching condition (\ref{bunching.condition}). Then the local unstable and local stable holonomy maps are {\em uniformly} $C^{1}$ when restricted to each center unstable and each center stable leaves respectively.  
\end{theorem}

\begin{theorem}\cite{brown2016}\label{brown} Let $f:M\to M$ be a $C^{1+\alpha}$ partially hyperbolic diffeomorphism that is dynamically coherent that satisfies the strong bunching condition (\ref{strong.bunching.condition}) for some $0<\theta<\alpha$. 
 Then the local unstable and local stable holonomy maps are {\em uniformly} $C^{1+\theta}$ when restricted to the center unstable and center stable leaves respectively. 
\end{theorem}

Let $W^{c}_{\eps}(x)$ denote the $c$-disc of radius $\eps>0$ and center $x$ inside the leaf $W^{c}(x)$ tangent to $E^{c}$.
\begin{proposition}\label{homogeneous} If $f$ satisfies the hypothesis of Theorem \ref{mteo:differentiability.AC}, then for each $x\in M$, the connected component of $AC(x)\cap W^{c}_{\eps}(x)$ containing $x$ is $C^{1}$-homogeneous.
\end{proposition}

\begin{proof}
Let $x,y\in AC(x)$. In order to prove the proposition it suffices to show that the homeomorphisms $h$ defined in Lemma~\ref{c.homogeneous} are $C^{1}$ when $U_{x}$ and $U_{y}$ are taken inside $W^{c}_{\eps}(x)$ and $W^{c}_{\eps}(y)$ respectively. \par
Let us assume that $y\in W^{s}_{\eps}(x)$, then the argument extends by applying it a finite number of times. If $f$ is in the hypothesis of Theorem \ref{mteo:differentiability.AC} then either $f$ is under the hypothesis of Theorem \ref{PSW}, in which case the local stable holonomy map is $C^{1}$ when restricted to $W^{sc}_{\eps}(x)$, or $f$ is under the hypothesis of Theorem \ref{brown}, in which case the local stable holonomy map is $C^{1+\text{H\"older}}$ when restricted to $W^{sc}_{\eps}(x)$. In either case, the corresponding $h$, which is exactly the stable holonomy map restricted to $W^{sc}_{\eps}(x)$ is $C^{1}$. 
\end{proof}

The following Proposition requires that $\dim E^{c}=2$.

\begin{proposition} If $f$ is under  the hypothesis of Theorem \ref{mteo:differentiability.AC}, then for each $x\in M$, the connected component of $AC(x)\cap W^{c}_{\eps}(x)$ is a $C^{1}$-manifold. 
\end{proposition}

The proof of this proposition follows immediately from Proposition \ref{homogeneous} above and the following theorem:

\begin{theorem} \cite{RSS1996} Let $K$ be a locally compact (possibly non-closed) subset of $\R^{n}$. Then $K$ is $C^1$-homogeneous if and only if $K$ is a $C^1$-submanifold of $ \mathbb{R}^n$.  
\end{theorem}
 
In order to prove Theorem \ref{mteo:differentiability.AC}, we will apply the following theorem by Journ\'e. It will be done in two steps. See discussion below Theorem \ref{journe}
  
\begin{theorem}\label{journe}\cite{J1988} Let ${\mathcal E}$ and $\mathcal{F}$ be two continuous transverse foliations with uniformly smooth leaves, of some manifold $N$, not necessarily compact. If the function $\phi:N\to \R^{m}$ is uniformly smooth along the leaves of $\mathcal{E}$ and $\mathcal{F}$, then $\phi$ is smooth.
\end{theorem}

We want to see that the immersions of the sets $BC(x)$, {\em the connected component of $AC(x)\cap B_{\eps}(x)$ containing $x$}, are smooth along transverse foliations. However, in our case we do not have smoothness along two transverse foliations {\em a priori}, so we proceed as Wilkinson in \cite{W2013}: first we prove that the immersions of $BC(x)$ are uniformly smooth along unstable and center leaves in a neighborhood of $x$. Since $\mathcal{F}^{u}$ and $\mathcal{F}^{c}$ are transverse foliations within each center-unstable leaf $W^{cu}(y)$, applying Journ\'e to each of these leaves we obtain smoothness along the foliation $\mathcal{F}^{cu}$. Note that this smoothness is uniform in a neighborhood of each $x$ (Proposition \ref{journe.cu}). \par

In this way we obtain two continuous transverse foliations, $\mathcal{F}^{s}$ and $\mathcal{F}^{cu}$, with uniformly smooth leaves, along whose leaves the immersion of $BC(x)$ is uniformly smooth. Journ\'e's Theorem \ref{journe} then implies that $BC(x)$ is smooth. This, in turn, implies that $AC(x)$ is an immersed $C^{1}$-manifold, and Theorem \ref{mteo:differentiability.AC} follows. 

\begin{proposition}\label{journe.cu} Let $f:M\to M$ be a dynamically coherent partially hyperbolic diffeomorphism. Let $\phi$ be uniformly smooth along the leaves of $\mathcal{F}^{c}$ and $\mathcal{F}^{u}$ in a neighborhood of $x\in M$, then $\phi$ is uniformly smooth along the leaves of $\mathcal{F}^{cu}$ in a neighborhood of $x$
\end{proposition}
\begin{proof}

This follows from Journ\'e's Theorem proof. A more detailed and reader-friendly proof may be found in \cite{KatokNitica2009}, Section 3, Preparatory results from analysis, and in particular, Section 3.3. Journ\'e's Theorem.
\end{proof}

So, in order to finish the proof of Theorem \ref{mteo:differentiability.AC}, the only thing left is to show that the immersions of the sets $BC(x)$ are uniformly smooth along the leaves of $\mathcal{F}^{c}$ in a neighborhood of each $x\in M$. This is proved in the following proposition:

\begin{proposition} 
The immersions of $BC(x)$ are uniformly smooth along the center leaves of $\mathcal{F}^{c}$.
\end{proposition}

\begin{proof} For each $y\in BC(x)$, call $C(y)=BC(x)\cap W^{c}_{\eps}(y)$. Now $C(y)$ is the image of $C(x)$ by the composition of the local stable holonomy restricted to $W^{cs}(x)$ and the local unstable holonomy restricted to $W^{cu}(y)$. Since  the local stable and unstable holonomies are uniformly $C^{1}$ when restricted to each center stable and center unstable leaves respectively, then $y\mapsto C(y)$ is uniformly $C^{1}$ in a neighborhood of $x$.
\end{proof}

\section{Some further properties of accessibility classes}

In this section, we give a further description of accessibility classes in our setting. Let us recall that $\phi:M\to 2^{M}$ is {\em lower semicontinuous} if for any open set $W\subseteq M$ the set of $x\in M$ such that $\phi(x)\cap W\ne\emptyset$ is open in $M$.
The following is a direct consequence of Lemma \ref{c.homogeneous}, and holds for center bundle of any dimension when there is dynamical coherence:
\begin{proposition}\label{lower.semicontinuity}
The function $x\mapsto AC(x)$ is lower semicontinuous. 
\end{proposition}
\begin{proof}
 Let $W$ be an open set, and let $x$ be such that $y\in AC(x)\cap W\ne \emptyset$. Take $U_{x}$ and $h$ as in Lemma \ref{c.homogeneous}, such that $h(U_{x})\subset W$. Then for all $w\in U_{x}$, $h(w)\in AC(w)\cap W\ne\emptyset$. Now $V=(\pi^{su})^{-1}(U_{x})$ is an open neighborhood of $x$, and for all $z\in V$, $AC(z)=AC(\pi^{su}(z))$, and $AC(\pi^{su}(z))\cap W\ne\emptyset$.
\end{proof}

The stable and unstable foliations are {\em jointly integrable at a point $x\in M$} if the stable holonomy map between two close center unstable disks, one of which contains $x$, takes the unstable leaf containing $x$ into another unstable leaf. Equivalently, if $F_{\eps}(x)=\{x\}$. Let $\Gamma (f)$ be the set of points $x$ of $M$ such that the stable and unstable foliations are jointly integrable at all points of $AC(x)$. Then $\Gamma (f)$ is a compact set laminated by the codimension two accessibility classes.

\begin{proposition} \label{perturbacion.clases}
Let $K\subset \Gamma(f)$ be a minimal set, meaning $K$ is a compact, $f$-invariant set, with no proper non-empty $f$-invariant subsets. Then there exists $f_{n}\to f$ in the $C^{r}$-topology, such that  $f_{n}|_{K}=f|_{K}$ and $AC_{f_{n}}(x)$ is either a codimension-one immersed manifold for all $x\in K$ or an open set for all $x\in K$.
\end{proposition}

Recall that in a minimal set, every orbit is dense, hence if $AC(x)$ is open for some $x\in K$, then all $AC(\xi)$ with $\xi\in K$ would be open, since the orbit of $\xi$ would eventually fall in $AC(x)$, what would imply $AC(f^{n}(x))$ is open for some $n\in\Z$, and $AC(f^{n}(\xi))=f^{n}(AC(\xi))$ is homeomorphic to $AC(\xi)$. 

Proposition \ref{perturbacion.clases} makes use of the following lemmas:

%lemmas which were proven in  \cite{HHU08,BHHTU}. Lemma \ref{keepaway.lema} was first proven by Ma\~n\'e for the case $W^{u}$ one-dimensional. 
%\begin{lemma}[Keepaway lemma] Let $\|Df^{-1}|_{E^{u}}\|<\mu^{-1}$, and let $N>0$ be such that $\mu^{N}>5$ and let $V$ be a small $s+c$-disc transverse to $E^{u}$
%such that 
%$$f^{n}(W^{u}_{5\eps}(V))\cap W^{u}_{\eps}(V)=\emptyset\qquad n=1,\dots, N; $$
%then for all $x\in M$ there exists $z\in W^{u}_{\eps}(x)$ such that $f^{n}\notin W^{u}_{\eps}(V)$ for all $n\in \N$.
%\end{lemma}
\begin{lemma}[\cite{HHU2008,BHHTU2008}]
The set of points that are non-recurrent in the future $\{z:z\notin\omega(z)\}$ is dense in every leaf of $\mathcal{F}^{u}$.  
\end{lemma}

\begin{lemma} \label{F.eps.no.vacio}
Let $K\subset \Gamma(f)$ be a minimal set (see definition in Proposition \ref{perturbacion.clases}). Then there exists $f_{n}\to f$ in the $C^{r}$-topology, such that  $f_{n}|_{K}=f|_{K}$ and $F_{f_{n},\eps}(x)\ne\{x\}$ for some $x\in K$. 
\end{lemma} 
The proof of Lemma \ref{F.eps.no.vacio} is Lemma 6.1. in \cite{BHHTU2008}. Indeed, what it is proved there is that the joint integrability at any $x\in K$ can be broken by a small $C^{r}$ push outside $K$. This yields openness of the accessibility class of $x$ of the perturbed system in the case $\dim E^{c}=1$, but only implies $F_{\eps,f_{n}}(x)\ne\{x\}$ in the case $\dim E^{c}=2$. 

\begin{lemma}\label{persistencia.F.eps} If $F_{\eps,f}(x)\ne\{x\}$, then there exist an open neighborhood $U(x)$of $x$ and a $C^{1}$-neighborhood $\mathcal{U}(f)$ of $f$ such that $F_{\eps,g}(y)\ne\{y\}$ for all $g\in \mathcal{U}$ and all $y\in U(x)$. Also, there exists $\mu>0$ such that $\text{diam }(F_{\eps,g}(y))>\mu$ for all $g\in\mathcal{U}$ and all $y\in U(x)$.
\end{lemma}
\begin{proof}
If $F_{\eps,f}(x)\ne\{x\}$, then there exists $z\in D_{x}$, $z\ne x$, such that $z\in \pi^{su}_{f}((\pi^{su})^{-1}(x))$. Let $B_{\rho}(z)\subset D_{x}$ be such that $x\notin B_{\rho}(z)$. Due to the continuity of $(x,f)\mapsto (W^{s}_{f}(x), W^{u}_{f}(x))$, there exist an open neighborhood $U(x)$ of $x$ and a $C^{1}$-neighborhood $\mathcal{U}(f)$ of $f$, such that $F_{\eps,g}(y)=\pi^{us}_{g}((\pi^{su})^{-1}(y))\cap B_{\rho}(z)\ne\emptyset$ for all $g\in\mathcal{U}$ and all $y\in U(x)$; therefore, $F_{\eps,g}(y)\ne\{y\}$ for all $g\in\mathcal{U}$ and all $y\in U(x)$. Also $\text{diam}\, F_{\eps,g}(y)>\mu>0$ for all $g\in\mathcal{U}$ and $y\in U(x)$, due to continuity.
 \end{proof}

Let us finish by putting all these concepts together. 
\begin{proof}[Proof of Proposition \ref{perturbacion.clases}] Let $K\subset \Gamma(f)$ be a minimal set as in the hypotheses. Then, by Lemma \ref{F.eps.no.vacio} there exists a sequence, call it $f_{n}\to f$, such that $f_{n}|_{K}=f|_{K}$ and $F_{f_{n},\eps}(x)\ne\{x\}$ for some $x_{n}\in K$. Now, by Lemma \ref{persistencia.F.eps}, there exists $\mu_{n}>0$ and an open neighborhood $U(x_{n})$ of $x_{n}$ such that $\text{diam}\, F_{\eps,f_{n}}(y)>\mu_{n}>0$ for all $y\in U(x_{n})$. Since $f_{n}|_{K}=f|_{K}$, $K$ is also a minimal set for $f_{n}$ for every $n$. Therefore, every $f_{n}$-orbit is dense in $K$. This implies that the $f_{n}$-orbit of every $\xi\in K$ enters the neighborhood $U(x_{n})$ and hence $\text{diam}\, F_{\eps,f_{n}}(\xi)>\mu_{n}>0$ for every $\xi$ in $K$. This implies that the stable and unstable manifolds of $\xi$ are not jointly integrable at any point of $K$. \par
So, $K\cap \Gamma(f_{n})=\emptyset$. In particular, $AC(\xi)$ cannot be a codimension-two immersed manifold for any $\xi\in K$. Now, if there were $\xi\in K$ such that $AC(\xi)$ is open then, by minimality, all $AC(\xi)$ with $\xi\in K$ would be open, by the comment right after Proposition \ref{perturbacion.clases}. So, if we assume there is an accessibility class that is not open in $K$, then for every point $\xi\in K$, $AC(\xi)$ has to be a codimension-one immersed manifold, as a result of Theorem \ref{mteo:class.manifold}.  
\end{proof}

\bibliographystyle{ijmart}

\end{document}